\documentclass{amsart}
\usepackage[style=alphabetic,backend=bibtex]{biblatex} 
\DeclareFieldFormat
  [article,inbook,incollection,inproceedings,patent,thesis,unpublished]
  {title}{\textit{#1}}

\usepackage{amssymb}
\usepackage{graphicx}
\usepackage[cmtip,all]{xy}
\usepackage{bbm}
\usepackage{comment}
\usepackage{mathrsfs}
\usepackage{eucal}
\usepackage{microtype}
\usepackage{xcolor}
\usepackage{hyperref}
\hypersetup{
 colorlinks,
 linkcolor={blue!50!black},
 citecolor={blue!50!black},
 urlcolor={blue!80!black}
}

\calclayout

\def\DefineSymbol#1#2{\newcommand{#1}{{\mathrm {#2}}}}
\def\DefineCategory#1#2{\newcommand{#1}{{\mathrm {#2}}}}

\theoremstyle{plain}
	\newtheorem{theorem}{Theorem}[section]
	\newtheorem{lemma}[theorem]{Lemma}
	
	\newtheorem{corollary}[theorem]{Corollary}
\theoremstyle{definition}
	\newtheorem{definition}[theorem]{Definition}
	\newtheorem{example}[theorem]{Example}

\theoremstyle{remark}
	
	\numberwithin{equation}{section}

\DefineSymbol{\pr}{pr}
\DefineSymbol{\id}{id}
\DefineSymbol{\const}{const}
\DefineSymbol{\op}{op}
\DefineSymbol{\diag}{diag}

\DeclareMathOperator{\Nm}{Nm}

\DeclareMathOperator{\Sym}{Sym}
\DeclareMathOperator{\Spec}{Spec}
\DeclareMathOperator{\End}{End}

\DefineCategory{\Set}{Set}
\DefineCategory{\Ab}{Ab}
\DefineCategory{\Mod}{Mod}
\DefineCategory{\Alg}{Alg}
\DefineCategory{\Ch}{Ch}
\DefineCategory{\Mon}{Mon}
\DefineCategory{\CMon}{CMon}

\DefineCategory{\Sm}{Sm}
\DefineCategory{\Cor}{Cor}
\DefineCategory{\QCor}{QCor}
\DefineSymbol{\red}{red}
\DefineSymbol{\tr}{tr}
\DeclareMathOperator{\ord}{ord}
\DeclareMathOperator{\length}{length}
\DeclareMathOperator{\Tr}{Tr}

\begin{document}

\title{Transfers on commutative group schemes}
\author{Junnosuke Koizumi}
\date{}
\address{Graduate School of Mathematical Sciences, University of Tokyo, 3-8-1 Komaba, Meguro-ku, Tokyo 153-8914, Japan}
\email{jkoizumi@ms.u-tokyo.ac.jp}
\maketitle
\begin{abstract}
We prove that any commutative group scheme separated over a noetherian normal scheme admits a canonical structure of a presheaf with transfers, which is characterized by a simple condition on radicial transfers.
\end{abstract}
\setcounter{tocdepth}{1}

\tableofcontents

\section*{Introduction}

Voevodsky's theory of presheaves with transfers (see for example \cite[Part 1]{MVW}) plays a central role in his theory of motives over a field.
Roughly speaking, they are presheaves on the category of smooth schemes equipped with covariant ``transfer maps'' for finite surjective morphisms, functorial in an appropriate sense.
One important example is commutative group schemes, e.g. $\mathbb{G}_a$, $\mathbb{G}_m$, the Witt ring scheme and abelian varieties.
In particular, smooth commutative group schemes over a field equipped with transfers are known to give examples of \emph{reciprocity sheaves} (see \cite{KSY1}, \cite{KSY2}), for which a nice motivic theory can be applied.

On the other hand, transfer structures on group schemes over a \emph{general base scheme} has not been studied much yet.
In this paper we prove the following general existence result.

\begin{theorem}[see Theorem \ref{thm:main}]\label{intro_main}
	Let $S$ be a noetherian normal scheme and $G$ be a separated commutative group scheme over $S$.
	Then there exists a canonical structure of a presheaf with transfers over $S$ on $G$.
\end{theorem}

One way to prove that a presheaf has a transfer structure is to show that it is a qfh sheaf.
In this way, Ancona-Huber-Lehalleur \cite{AHL} proved that for any smooth commutative group scheme $G$ over a noetherian excellent scheme $S$, the presheaf $G_\mathbb{Q}=G\otimes_\mathbb{Z}\mathbb{Q}$ admits a unique transfer structure.
However, being a qfh sheaf is so strong a condition that we cannot expect in general that $G$ itself should be so.
Also, transfers constructed in this way are comparatively inexplicit and difficult to compute.

Another way is to use the symmetric product and construct transfers geometrically.
In this way, Spie\ss-Szamuely \cite{SS03} and Barbieri-Viale-Kahn \cite[Lemma 1.4.4]{BVK16} proved that any commutative group scheme locally of finite type over a field admits a canonical transfer structure.
Our proof of Theorem \ref{intro_main} is based on this idea.
In section 2, we imitate the construction of locally free transfers due to Spie\ss-Szamuely in a more general setting.
In section 3, we define the canonical transfer on group schemes and prove the functoriality to establish Theorem \ref{intro_main}.
In section 4, we characterize the canonical transfer by a simple condition on radicial transfers.

\subsection*{Acknowledgements}
I am grateful to Shuji Saito for his support in my studies.
I also thank Hiroyasu Miyazaki for many interesting discussions.

\section{Review of relative cycles}

In this section we recall the theory of relative cycles in the style of Cisinski-D\'eglise \cite{CD}; nothing in this section is our original.
Let $X$ be a noetherian scheme.
A {\it cycle} on $X$ is a formal $\mathbb{Z}$-linear combination of integral closed subschemes of $X$.
A {\it component} of a cycle $\alpha$ is an integral closed subscheme whose coefficient in $\alpha$ is non-zero.
For a closed subscheme $W$ of $X$, we define the associated cycle $[W]$ by
$$
	\textstyle [W]=\sum_{i=1}^n \length(\mathcal{O}_{X,W_i})[W_i]
$$
where $W_1,\dots,W_n$ are the irreducible components of $W$.
If $f\colon X\to Y$ is a morphism between noetherian schemes and $\alpha$ is a cycle on $X$, then the cycle $f_*\alpha$ on $Y$ is defined by linearly extending
$$
	[V]\mapsto
	\begin{cases}
		[k(\xi):k(f(\xi))][\overline{\{f(\xi)\}}]		&([k(\xi):k(f(\xi))]<\infty)\\
		0							&(\text{otherwise})
	\end{cases}
$$
where $\xi$ is the generic point of $V$.

Until the end of this section, we fix a noetherian base scheme $S$ and consider only noetherian schemes over $S$.

\begin{definition}

Let $X$ be a noetherian $S$-scheme and $\alpha$ a cycle on $X$.
\begin{itemize}
	\item	We say that $\alpha$ is {\it finite} over $S$ if every component of $\alpha$ is finite over $S$.
	\item	We say that $\alpha$ is {\it flat} over $S$ if every component of $\alpha$ is flat over $S$.
	\item	We say that $\alpha$ is {\it pseudo-dominant} over $S$ if every component of $\alpha$ is dominant over some irreducible component of $S$.
\end{itemize}
\end{definition}

Consider the following diagram of noetherian schemes.
\begin{align}\label{pullback_diagram}
	\xymatrix{
			&X\ar[d]\\
	Y\ar[r]	&S
	}
\end{align}
Let $\alpha$ (resp. $\beta$) be a cycle on $X$ (resp. $Y$).
If $\alpha$ is {\it $\mathbb{Z}$-universal} over $S$ in the sense of \cite{CD}, then a cycle $\alpha\otimes_S \beta$ on $X\times_SY$ called the {\it pullback} of $\alpha$ by $\beta$ is defined.
The operation ${-}\otimes_S{-}$ is bilinear.

\begin{lemma}\label{lem:cycle_property}
	Let $X$ be a noetherian $S$-scheme and $\alpha$ be a cycle on $X$.
	Then $\alpha$ is finite and flat over $S$ $\implies$ $\mathbb{Z}$-universal over $S$ $\implies$ pseudo-dominant over $S$.
\end{lemma}

\begin{proof}
	This follows directly from the definition of $\mathbb{Z}$-universal cycles; see \cite[Definition 8.1.47]{CD}.
\end{proof}

\begin{lemma}\label{lem:pullback_lemma}
	Consider the diagram (\ref{pullback_diagram}) of noetherian schemes.
	Let $\alpha$ be a cycle on $X$ finite and flat over $S$ and write $\alpha = \sum_{i=1}^n m_i[V_i]$.
	Then we have
	$$
		\textstyle\alpha\otimes_S[Y] = \sum_{i=1}^{n}m_i[V_i\times_SY].
	$$
\end{lemma}

\begin{proof}
		See \cite[8.1.35 (P3)]{CD}.
\end{proof}

\begin{lemma}\label{flattening}
	Suppose that $S$ is reduced.
	Let $X$ be a noetherian $S$-scheme and $\alpha$ be a cycle on $X$ $\mathbb{Z}$-universal over S.
	Then there exists a dominant blow-up $p\colon S'\to S$ such that $\alpha\otimes_S[S']$ is flat over $S'$.
\end{lemma}

\begin{proof}
	See \cite[Lemma 8.1.18 and 8.1.35 (P5)]{CD}.
\end{proof}

\begin{definition}
For a noetherian $S$-scheme $X$, we define $c_0(X/S)$ to be the abelian group of cycles on $X$ finite and $\mathbb{Z}$-universal over $S$.
An element of $c_0(X/S)$ is called a \emph{relative $0$-cycle} on $X$ over $S$.
\end{definition}

Let $X,Y,Z$ be noetherian $S$-schemes.
We set $c_S(X,Y):=c_0(X\times_SY/X)$ and call its elements {\it finite correspondences} from $X$ to $Y$ over $S$.
For example, the graph $\Gamma_f$ of a morphism $f\colon X\to Y$ over $S$ gives an element $[\Gamma_f] \in c_S(X,Y)$, for which we simply write $f$.
For $\alpha\in c_S(X,Y)$ and $\beta\in c_S(Y,Z)$, we define a cycle $\beta\circ\alpha$ on $X\times_SZ$ by the formula
$$
	\beta\circ\alpha = {\pr_{13}}_*(\beta\otimes_Y\alpha)
$$
where ${\pr_{13}}_*\colon X\times_S Y\times_S Z\to X\times_S Z$ is the canonical projection.
By the definition of pullback (cf. \cite[Theorem 8.1.39]{CD}) and \cite[Corollary 8.2.6]{CD}, we have $\beta\otimes_Y\alpha\in c_0(X\times_SY\times_SZ/X)$.
If $Z$ is separated and of finite type over $S$, then it follows from \cite[Section 9.1.1]{CD} that $\beta\circ \alpha \in c_S(X,Z)$.
This gives a bilinear pairing
$$
	\circ\colon c_S(X,Y)\times c_S(Y,Z)\to c_S(X,Z).
$$

\begin{lemma}\label{lem:composition_formula}
	Let $S$ be a noetherian scheme and $X,Y,Z,W$ be noetherian $S$-schemes.
	\begin{enumerate}
		\item	Let $\alpha\in c_S(X,Y)$, $\beta\in c_S(Y,Z)$ and $\gamma \in c_S(Z,W)$.
				Suppose that $Z,W$ are separated and of finite type over $S$.
				Then we have $$(\gamma\circ \beta)\circ \alpha = \gamma\circ (\beta\circ \alpha).$$
		\item	Let $f\colon X\to Y$ be an $S$-morphism and $\beta\in c_S(Y,Z)$.
				Suppose that $Z$ is separated and of finite type over $S$.
				Then we have $$\beta\circ f = \beta\otimes_Y[X].$$
	\end{enumerate}
\end{lemma}

\begin{proof}
	See \cite[Proposition 9.1.7]{CD}.
	In loc.cit., it is assumed that $X,Y,Z,W$ are all separated and of finite type over $S$, but the proof works verbatim.
\end{proof}

Let $\Sm_S$ denote the category of smooth separated $S$-schemes of finite type.
Using the pairing above as composition, we can define an additive category $\Cor_S$ whose objects are the same as $\Sm_S$ and the morphisms are finite correspondences.
The category $\Sm_S$ can be embedded into $\Cor_S$ by the graph construction.
A presheaf $\Cor_S^\op\to \Ab$ is called a \emph{presheaf with transfers} over $S$.

\section{Locally free transfers}

In this section we construct transfers for finite locally free morphisms.

Let $A$ be a ring and $B$ a finite $A$-algebra.
We write $B^{\odot d}$ for the $A$-algebra $(B^{\otimes d})^{S_d}$, i.e. the subalgebra of $B^{\otimes d}$ fixed under the canonical action of the symmetric group $S_d$.

Suppose that $B$ is free of rank $d$ as an $A$-module.
Let $e_1,\dots,e_d$ be a basis of $B$ over $A$ and let $I_{[m,n]}^d$ denote the set of $S_d$-orbits in $\{m,m+1,\dots,n\}^d$.
For an orbit $\Gamma\in I_{[1,d]}^d$, we set
$$
	\textstyle e_\Gamma := \sum_{(i_1,\dots,i_d)\in \Gamma}e_{i_1}\otimes \dots\otimes e_{i_d}.
$$
Then $\{e_\Gamma\}_{\Gamma\in I_{[1,d]}^d}$ is a basis of $B^{\odot d}$ over $A$.
We define an action of $B^{\odot d}$ on $\wedge^d B$ by
$$
	\textstyle e_\Gamma\cdot (b_1\wedge\dots\wedge b_d) = \sum_{(i_1,\dots,i_d)\in \Gamma} e_{i_1}b_1\wedge\dots\wedge e_{i_d}b_d.
$$
This is well-defined since we are taking a sum over all possible permutations of $(i_1,\dots,i_d)$.
In this way we obtain a morphism of $A$-algebras
$$
	u\colon B^{\odot d}\to \End_A(\wedge^dB)\simeq A.
$$
One can easily verify that this does not depend on the choice of the basis.

\begin{lemma}\label{lem:reduction}
	Let $K$ be a field and $B$ be a finite local $K$-algebra of dimension $d$ with residue field $K$.
	Let $\varphi$ denote the composition $B^{\otimes d}\xrightarrow{\cdot} B\twoheadrightarrow K$. 
	Then the restriction of $\varphi$ to $B^{\odot d}$ equals $u\colon B^{\odot d}\to K$.
\end{lemma}

\begin{proof}
	Let $J$ denote the maximal ideal of $B$.
	First note that $J$ is a nilpotent ideal since $B$ is artinian local.
	Take a basis $e_1,\dots,e_d$ of $B$ over $K$ so that $e_d=1$ and $J^i$ is spanned by $e_1,\dots,e_{d_i}$ as a $K$-vector space, where $d_i=\dim_K J^i$.
	For $\Gamma\in I^d_{[1,d]}$ we have
	$$
		\varphi(e_\Gamma) =
		\begin{cases}
			1		&(\Gamma = \{(d,\dots,d)\})\\
			0		&(\text{otherwise})
		\end{cases}
	$$
	since $J$ is an ideal.
	For an element $b$ of $B$, we define its order $\ord(b)$ to be the maximum value of $i$ satisfying $b\in J^i$ (we set $\ord(0)=\infty$).
	Let $b_1,\dots,b_d\in B$ satisfy $\ord(b_i)\geq \ord(b_i)$ for all $i$ and $\ord(b_j) > \ord(b_j)$ for some $j$.
	Then the number of $b_i$'s contained in $J^k$ exceeds the dimension of $J^k$ for some $k$, so we get $b_1\wedge\dots\wedge b_d=0\in \wedge^d B$.
	Using this fact, the action of  $B^{\odot d}$ on $\wedge^dB$ can be computed as
	$$
		e_\Gamma\cdot (e_1\wedge \dots\wedge e_d) =
		\begin{cases}
			e_1\wedge\dots\wedge e_d	&(\Gamma = \{(d,\dots,d)\})\\
			0		&(\text{otherwise}).
		\end{cases}
	$$
	This completes the proof.
\end{proof}

\begin{lemma}\label{lem:p_well_defined}
	Let $A$ be a ring and $B_i\;(i=1,2)$ be finite free $A$-algebras of rank $d_i$.
	The morphism of $A$-algebras
	$$
		\pr_1^{\otimes d_1}\otimes \pr_2^{\otimes d_2}\colon (B_1\times B_2)^{\otimes (d_1+d_2)}\to B_1^{\otimes d_1}\otimes_A B_2^{\otimes d_2}
	$$
	restricts to a morphism
	$$
		p\colon (B_1\times B_2)^{\odot (d_1+d_2)}\to B_1^{\odot d_1}\otimes_A B_2^{\odot d_2}
	$$
	and fits into the following commutative diagram.
	$$
	\xymatrix{
		(B_1\times B_2)^{\odot (d_1+d_2)}\ar[r]^-u	\ar[d]^-p			&A\ar@{=}[d]\\
		B_1^{\odot d_1}\otimes_A B_2^{\odot d_2}\ar[r]^-{u\otimes u}			&A
	}
	$$
\end{lemma}

\begin{proof}
	Let $e_1,\dots,e_{d_1}$ (resp. $e_{d_1+1},\dots,e_{d_1+d_2}$) be a basis of $B_1$ (resp. $B_2$) over $A$.
	For $\Gamma_1\in I_{[1,d_1]}^{d_1}$ and $\Gamma_2\in I_{[d_1+1,d_1+d_2]}^{d_2}$ we define $\Gamma_1\ast\Gamma_2\in I_{[1,d_1+d_2]}^{d_1+d_2}$ by concatenation.
	Then for an orbit $\Gamma \in I_{[1,d_1+d_2]}^{d_1+d_2}$ we have
	\begin{align*}
				(\pr_1^{\otimes d_1}\otimes \pr_2^{\otimes d_2})(e_\Gamma) =
				\begin{cases}
					e_{\Gamma_1}\otimes e_{\Gamma_2}		&	(\Gamma = \Gamma_1\ast\Gamma_2 \text{ for some }\Gamma_1\in I_{[1,d_1]}^{d_1},\;\Gamma_2\in I_{[d_1+1,d_1+d_2]}^{d_2})\\
					0									&	(\text{otherwise}).
				\end{cases}
	\end{align*}
	This proves the first assertion.
	In the first case we have
	\begin{align*}
			e_\Gamma\cdot (e_1\wedge\dots\wedge e_{d_1+d_2})
		{}={}&\textstyle\sum_{(i_1,\dots,i_{d_1+d_2})\in \Gamma} e_{i_1}e_1\wedge\dots\wedge e_{i_{d_1+d_2}}e_{d_1+d_2}\\
		{}={}&\textstyle\sum_{(i_1,\dots,i_{d_1})\in \Gamma_1} \sum_{(i_{d_1+1},\dots,i_{d_1+d_2})\in \Gamma_2} e_{i_1}e_1\wedge\dots\wedge e_{i_{d_1+d_2}}e_{d_1+d_2}\\
		{}={}&(e_{\Gamma_1}\cdot (e_1\wedge\dots\wedge e_{d_1}))\wedge (e_{\Gamma_2}\cdot (e_{d_1+1}\wedge\dots\wedge e_{d_1+d_2}))
	\end{align*}
	In the second case the action of $e_\Gamma$ on $\wedge^{d_1+d_2}(B_1\times B_2)$ is trivial.
	These results imply the commutativity of the diagram.
\end{proof}

Let $S$ be a scheme.
For a finite morphism $f\colon Y\to X$ between $S$-schemes, the symmetric product $\Sym^d_X(Y)$ is defined to be the quotient of $Y\times_X\dots\times_XY$ ($d$ times) by the canonical action of the symmetric group $S_d$, i.e.
$$
	\Sym^d_X(Y)=\underline{\Spec}_X(f_*\mathcal{O}_Y)^{\odot d}
$$
(see \cite[Expos\'e V, 1]{SGA1} for basic facts about quotients of schemes by finite groups).

Suppose that $f$ is a finite locally free morphism of constant rank $d$.
The morphism of $\mathcal{O}_X$-algebras $u\colon (f_*\mathcal{O}_Y)^{\odot d}\to \mathcal{O}_X$ (see the preamble to Lemma \ref{lem:reduction}) gives rise to a morphism $f^\sharp\colon X\to \Sym^d_X(Y)$ (note that this morphism is considered also in \cite[Section 6]{SV96}).
Now let $G$ be a commutative group scheme over $S$.
For any $g\in G(Y)$, the morphism
$$
	Y\times_X\dots \times_X Y\xrightarrow{(g,\dots,g)} G\times_S\dots\times_S G\xrightarrow{+} G
$$
is $S_d$-invariant, so it descends to a morphism $\sigma(g)\colon \Sym_X^d(Y)\to G$.
We define $f_*g\in G(X)$ to be the composition
$$
	X\xrightarrow{f^\sharp}\Sym_X^d(Y)\xrightarrow{\sigma(g)}G.
$$

\begin{lemma}\label{lem:field_norm_trace}
	Let $S$ be a scheme, $K$ be a field over $S$ and $L/K$ be a finite extension.
	Let $f\colon \Spec L\to \Spec K$ be the corresponding morphism.
	Then we have
	\begin{align*}
		f_*{}={}&\Nm_{L/K}\colon \mathbb{G}_{m,S}(L)\to \mathbb{G}_{m,S}(K),\\
		f_*{}={}&\Tr_{L/K}\colon \mathbb{G}_{a,S}(L)\to \mathbb{G}_{a,S}(K).
	\end{align*}
\end{lemma}

\begin{proof}
	Let $e_1,\dots,e_d$ be a basis of $L$ over $K$.
	Let $a\in \mathbb{G}_{m,S}(L)=L^\times$.
	By definition, the action of $a^{\otimes d}\in L^{\odot d}$ on $\wedge^dL$ is equal to the scalar multiplication by $f_*a$.
	On the other hand, we have
	$$
		a^{\otimes d}\cdot (e_1\wedge \dots\wedge e_d) = ae_1\wedge\dots\wedge ae_d = \operatorname{det}_{K} (L\xrightarrow{a}L)\cdot(e_1\wedge \dots\wedge e_d),
	$$
	so we get $f_*a=\operatorname{det}_{K} (L\xrightarrow{a}L) = \Nm_{L/K}(a)$.
	The proof of the second equality is similar.
\end{proof}

\begin{lemma}\label{lem:base_change}
	Let $S$ be a scheme, $f\colon Y\to X$ be a finite locally free morphism of constant degree $d$ between $S$-schemes and $h\colon X'\to X$ be a morphism of $S$-schemes.
	Consider the following Cartesian diagram.
	$$
	\xymatrix{
		Y'\ar[r]^-{h'}\ar[d]^-{f'}		&Y\ar[d]^-{f}\\
		X'\ar[r]^-{h}					&X
	}
	$$
	For any commutative group scheme $G$ over $S$, we have
	$$
		h^*f_*=f'_*{h'}^*\colon G(Y)\to G(X').
	$$
\end{lemma}

\begin{proof}
	Since $Y$ and $\Sym^d_X(Y)$ are finite locally free over $X$, all the constructions we used to define $f_*$ are compatible with base-change, so the claim is obvious.
\end{proof}

\begin{lemma}\label{lem:reduction_scheme}
	Let $S$ be a scheme, $K$ a field over $S$ and $f\colon X\to \Spec K$ a finite morphism of degree $d$.
	Suppose that $X$ is connected and has a section $s\colon \Spec K\to X$.
	Then for any commutative group scheme $G$ over $S$, we have
	$$
		f_*=d\cdot s^*\colon G(X)\to G(K).
	$$
\end{lemma}

\begin{proof}
	Let $g\in G(X)$.
	Consider the following diagram of $S$-schemes.
	$$
	\xymatrix@C=40pt{
		\Spec K\ar[r]^-{s}\ar[drr]^-{f^\sharp}	&X\ar[r]^-{\Delta}		&X\times_K\dots\times_KX\ar[r]^-{g\times\dots\times g}\ar[d]		&G\times_S\dots\times_SG\ar[d]^-{+}\\
					&		&\Sym_K^d(X)	\ar[r]^-{\sigma(g)}				&G
	}
	$$
	The left triangle is commutative by Lemma \ref{lem:reduction} and the right square is commutative by the definition of $\sigma(g)$.
	Therefore the total trapezoid is commutative, which implies $f_*g = d\cdot s^*g$.
\end{proof}

\begin{lemma}\label{lem:coproduct}
	Let $S$ be a scheme and $f_i\colon Y_i\to X\;(i=1,2)$ be finite locally free morphisms of constant degree $d_i$ between $S$-schemes.
	Let $G$ be a commutative group scheme over $S$.
	For any $g_1\in G(Y_1)$ and $g_2\in G(Y_2)$, we have
	$$
		{f_1}_*g_1+{f_2}_*g_2 = (f_1,f_2)_*(g_1,g_2)
	$$
	in $G(X)$, where $(f_1,f_2)\colon Y_1\sqcup Y_2\to X$ and $(g_1,g_2)\colon Y_1\sqcup Y_2\to G$ are morphisms induced by the universal property of coproducts.
\end{lemma}

\begin{proof}
	Consider the following diagram of $S$-schemes.
	$$
	\xymatrix@C=40pt{
		X\ar[r]^-{(f_1^\sharp,f_2^\sharp)}\ar@{=}[d]			&\Sym^{d_1}_X(Y_1)\times_X\Sym^{d_2}_X(Y_2)\ar[r]^-{\sigma(g_1)\times\sigma(g_2)}\ar[d]^-\pi		&G\times_S G\ar[d]^-{+}\\
		X\ar[r]^-{(f_1,f_2)^\sharp}							&\Sym^{d_1+d_2}_X(Y_1\sqcup Y_2)\ar[r]^-{\sigma(g_1,g_2)}						&G
	}
	$$
	Here, the morphism $\pi$ is induced by the morphism $p$ given in Lemma \ref{lem:p_well_defined}.
	The left square is commutative by Lemma \ref{lem:p_well_defined} and the right square is commutative by the definition of $\sigma$.
	Therefore the total rectangle is commutative, which implies the desired equality.
\end{proof}

\begin{corollary}\label{cor:alg_cld}
	Let $S$ be a scheme, $K$ an algebraically closed field over $S$ and $f\colon X\to \Spec K$ a finite morphism.
	Write $X=\{x_1,\dots,x_n\}$ and set $d_i = \dim_K \mathcal{O}_{X,x_i}$.
	Let $p_i\colon \{x_i\}\xrightarrow{\sim} \Spec K$ be the canonical projection and $r_i\colon \{x_i\}\to X$ be the canonical closed immersion.
	Then for any commutative group scheme $G$ over $S$, we have
	$$
		\textstyle f_*=\sum_{i=1}^n d_i\cdot {p_i}_*r_i^*\colon G(X)\to G(K).
	$$
\end{corollary}

\begin{proof}
	This follows immediately from Lemma \ref{lem:reduction_scheme} and Lemma \ref{lem:coproduct}.
\end{proof}

\section{The canonical transfer}

In this section we fix a noetherian base scheme $S$ and a separated commutative group scheme $G$ over $S$.
First we note the following fact.

\begin{lemma}\label{lem:dominant_injective}
	Let $f\colon X\to Y$ be a dominant morphism between $S$-schemes.
	If $Y$ is reduced, then $f^*\colon G(Y)\to G(X)$ is injective.
\end{lemma}

\begin{proof}
	Suppose that $g,h\colon Y\to G$ are two morphisms over $S$ and $f^*g=f^*h$.
	Then $(g,h)\colon Y\to G\times_SG$ sends the generic points of $Y$ to points inside the diagonal $\Delta_G$.
	Since $G$ is separated over $S$, the diagonal $\Delta_G$ is closed in $G\times_SG$ and hence the image of $Y$ lies in $\Delta_G$ set-theoretically.
	Since $Y$ is reduced, this morphism factors through $\Delta_G$ scheme-theoretically, i.e. $g=h$.
\end{proof}

Let $X,Y$ be noetherian integral schemes over $S$ and $\alpha \in c_S(X,Y)$.
We construct the \emph{canonical transfer} $\alpha^*\colon G(Y)\to G(X)$ in the following two cases:
\begin{enumerate}
	\item	$\alpha$ is flat over $X$.
	\item	$X$ is normal.
\end{enumerate}

First suppose that $\alpha$ is flat over $X$.
Write $\alpha = \sum_{i=1}^nm_i[V_i]$.
Then each $V_i$ is finite locally free of constant rank over $X$.
Let $p_i\colon V_i\to X$ and $q_i\colon V_i\to Y$ be canonical projections.
We define
$$
	\textstyle \alpha^* = \sum_{i=1}^n m_i {p_i}_*q_i^*\colon G(Y)\to G(X).
$$
Note that for an $S$-morphism $f\colon X\to Y$, this definition of $f^*$ coincides with the usual pullback.

Next suppose that $X$ is normal.
Take a dense open subset $U\subset X$ such that $\alpha|_U$ is flat over $U$.
By the next lemma, for any $g\in G(Y)$ the element $(\alpha|_U)^*g\in G(U)$ lies in the image of $G(X)\to G(U)$.
We define $\alpha^*\colon G(Y)\to G(X)$ by setting $\alpha^*g:=(\alpha|_U)^*g$; this does not depend on the choice of $U$.

\begin{lemma}
	Let $X$ be a noetherian normal integral scheme, $U\subset X$ be a dense open subset and $f\colon V\to X$ a finite morphism.
	Set $V_U=V\times_XU$ and suppose that $f_U\colon V_U\to U$ is finite locally free of constant rank $d$.
	Then for any $g\in G(V)$, the element ${f_U}_*(g|_{V_U})\in G(U)$ lies in the image of $G(X)\to G(U)$.
\end{lemma}

\begin{proof}
	Recall that ${f_U}_*(g|_{V_U})$ is defined to be the composition
	$$
		U\xrightarrow{f_U^\sharp}	\Sym^d_U(V_U)\xrightarrow{\sigma(g|_{V_U})}	G.
	$$
	By construction, we have $\Sym^d_U(V_U) \simeq \Sym^d_X(V)\times_XU$.
	Since $X$ is noetherian, $\Sym^d_X(V)$ is finite over $X$.
	Since $X$ is normal, $f^\sharp_U$ uniquely extends to a morphism $f^\sharp\colon X\to \Sym_X^d(V)$.
	Then ${f_U}_*(g|_{V_U})$ is the image of
	$$
		X\xrightarrow{f^\sharp} \Sym^d_X(V)\xrightarrow{\sigma(g)} G
	$$
	under $G(X)\to G(U)$.
\end{proof}

\begin{example}
	Let $X,Y$ be noetherian integral schemes over $S$ and $\alpha\in c_S(X,Y)$.
	Suppose that $X$ is normal.
	Write $\alpha = \sum_{i=1}^n m_i[V_i]$ and let $q_i\colon V_i\to Y$ be the canonical projection.
	Then the canonical transfer $\alpha^*\colon \mathbb{G}_{m,S}(Y)\to \mathbb{G}_{m,S}(X)$ defined above is given by
	$$
		\textstyle \alpha^*g = \prod_{i=1}^n (\Nm_{k(V_i)/k(X)} (q_i^*g))^{m_i}.
	$$
	This follows from Lemma \ref{lem:field_norm_trace} and Lemma \ref{lem:base_change}.
	Similarly, the canonical transfer $\alpha^*\colon \mathbb{G}_{a,S}(Y)\to \mathbb{G}_{a,S}(X)$ is given by
	$$
		\textstyle \alpha^*g = \sum_{i=1}^n m_i \Tr_{k(V_i)/k(X)} (q_i^*g).
	$$
\end{example}

In the following, we will prove that the canonical transfer defined above is functorial, i.e. $\alpha^*\beta^*=(\beta\circ\alpha)^*$.
We start with a special case.

\begin{lemma}\label{flat_arbitrary}
	Let $X,Y$ be noetherian integral schemes over $S$ and $\alpha \in c_S(X,Y)$.
	Suppose that $\alpha$ is flat over $X$ and $Y$ is separated and of finite type over $S$.
	Let $K$ be an algebraically closed field and $\rho\in c_S(\Spec K,X)$.
	Then we have
	$$
		\rho^*\alpha^* = (\alpha\circ \rho)^*\colon G(Y)\to G(K).
	$$
\end{lemma}

\begin{proof}
	Firstly, $\rho\in c_S(\Spec K,X)$ is a formal $\mathbb{Z}$-linear combination of $K$-rational points on $\Spec K\times_SX$.
	By linearity we may assume that $\rho$ is a \emph{morphism} from $\Spec K$ to $X$ over $S$.
	Write $\alpha = \sum_{i=1}^nm_i[V_i]$.
	Then each $V_i$ is finite and flat over $X$ and hence finite locally free of constant rank over $X$.
	Consider the following diagram, where the left square is Cartesian.
	$$
	\xymatrix{
		V_i\times_X\Spec K\ar[r]^-{\rho'_i}\ar[d]^-{p'_i}		&V_i\ar[d]^-{p_i}\ar[r]^-{q_i}		&Y\\
		\Spec K\ar[r]^-{\rho}							&X
	}
	$$
	By Lemma \ref{lem:base_change}, we have
	$$
		\textstyle \rho^*\alpha^* = \sum_{i=1}^n m_i \rho^*{p_i}_*q_i^* = \sum_{i=1}^n m_i {p'_i}_*{\rho'_i}^*q_i^*.
	$$
	On the other hand, since $\alpha$ is flat over $X$, we have
	\begin{align*}
		\textstyle\alpha\circ\rho	= \alpha\otimes_X\Spec K =	\textstyle\sum_{i=1}^n m_i[V_i\times_X\Spec K]	
	\end{align*}
	by Lemma \ref{lem:pullback_lemma}.
	Write $[V_i\times_X\Spec K] = \sum_{j=1}^{n_i} m_{ij}[\xi_{ij}]$.
	Let $p_{ij}\colon \{\xi_{ij}\}\to \Spec K$ be the canonical projection and $r_{ij}\colon \{\xi_{ij}\}\to V_i\times_X\Spec K$ be the canonical closed immersion.
	Then we have
	$$
		\textstyle (\alpha\circ \rho)^* = \sum_{i=1}^n m_i \sum_{j=1}^{n_i} m_{ij} {p_{ij}}_*r_{ij}^*{\rho_i'}^*q_i^*.
	$$
	Therefore it suffices to show that
	$$
		\textstyle {p'_i}_* = \sum_{j=1}^{n_i} m_{ij}{p_{ij}}_*r_{ij}^*
	$$
	holds for each $i$.
	This follows from Corollary \ref{cor:alg_cld}.
\end{proof}

\begin{lemma}\label{normal_generic}
	Let $X,Y$ be noetherian integral schemes over $S$ and $\alpha \in c_S(X,Y)$.
	Suppose that $X$ is normal and $Y$ is separated and of finite type over $S$.
	Let $K$ be an algebraically closed field and $\rho\colon \Spec K\to X$ be a morphism onto the generic point of $X$.
	Then we have
	$$
		\rho^*\alpha^* = (\alpha\circ \rho)^*\colon G(Y)\to G(K).
	$$
\end{lemma}

\begin{proof}
	Let $U\subset X$ be a dense open subset such that $\alpha|_U$ is flat over $U$.
	Let $j\colon U\to X$ be the canonical open immersion and $\rho'\colon \Spec K\to U$ be the restriction of $\rho$.
	Consider the following diagram.
	$$
		\xymatrix@C=40pt{
			G(Y)\ar@{=}[r]\ar[d]^-{\alpha^*}\ar@/^15pt/[rr]^-{(\alpha\circ\rho)^*}	&G(Y)\ar[d]_-{(\alpha|_U)^*}\ar[r]	_-{(\alpha|_U\circ \rho')^*}					&G(K)\ar@{=}[d]\\
			G(X)\ar[r]^-{j^*}\ar@/_15pt/[rr]_-{\rho^*}					&G(U)\ar[r]^-{{\rho'}^*}								&G(K)
		}
	$$
	Let us verify that all faces are commutative.
	Two triangles on the top and the bottom are clearly commutative.
	The commutativity of the left square follows from the definition of $\alpha^*\colon G(Y)\to G(X)$.
	The commutativity of the right square follows from Lemma \ref{flat_arbitrary}.
	Therefore we get $\rho^*\alpha^* = (\alpha\circ \rho)^*$.
\end{proof}

\begin{lemma}\label{normal_dominant}
	Let $X,Y,X'$ be noetherian integral schemes over $S$ and $\alpha \in c_S(X,Y)$.
	Suppose that $X$ is normal and $Y$ is separated and of finite type over $S$.
	Let $f\colon X'\to X$ be a dominant morphism such that $\alpha\circ f$ is flat over $X'$.
	Then we have
	$$
		f^*\alpha^* = (\alpha\circ f)^*\colon G(Y)\to G(X').
	$$
\end{lemma}

\begin{proof}
	Let $K$ be an algebraic closure of the function field of $X'$ and $\rho\colon \Spec K\to X'$ be the canonical morphism.
	Then we have
	\begin{align*}
		\rho^*(\alpha\circ f)^*	&{}={} (\alpha\circ f\circ \rho)^*		&\text{(by Lemma \ref{flat_arbitrary})}\\
								&{}={} (f\circ \rho)^*\alpha^*			&\text{(by Lemma \ref{normal_generic})}\\
								&{}={} \rho^*f^*\alpha^*.
	\end{align*}
	Since $\rho^*\colon G(X')\to G(K)$ is injective, this completes the proof.
\end{proof}

\begin{lemma}\label{normal_arbitrary}
	Let $X,Y$ be noetherian integral schemes over $S$ and $\alpha \in c_S(X,Y)$.
	Suppose that $X$ is normal and $Y$ is separated and of finite type over $S$.
	Let $K$ be an algebraically closed field and $\rho\in c_S(\Spec K, X)$.
	Then we have
	$$
		\rho^*\alpha^* = (\alpha\circ \rho)^*\colon G(Y)\to G(K).
	$$
\end{lemma}

\begin{proof}
	As in the proof of Lemma \ref{flat_arbitrary}, we may assume that $\rho$ is a morphism $\Spec K\to X$ over $S$.
	By Lemma \ref{flattening}, there is a dominant blow-up $\pi\colon X'\to X$ such that $\alpha\circ \pi = \alpha\otimes_X[X']$ is flat over $X'$.
	Since $K$ is algebraically closed, we can lift $\rho$ to $\rho'\colon \Spec K\to X'$.
	Then we have
	\begin{align*}
		\rho^*\alpha^*	&{}={} {\rho'}^*\pi^*\alpha^*\\
						&{}={} {\rho'}^*(\alpha\circ \pi)^*		&\text{(by Lemma \ref{normal_dominant})}\\
						&{}={} (\alpha\circ \pi \circ \rho')^*	&\text{(by Lemma \ref{flat_arbitrary})}\\
						&{}={} (\alpha\circ \rho)^*.
	\end{align*}
	This completes the proof.
\end{proof}

\begin{theorem}\label{thm:main}
	Let $X,Y,Z$ be noetherian integral schemes over $S$ and $\alpha \in c_S(X,Y)$, $\beta\in c_S(Y,Z)$.
	Suppose that $X,Y$ are normal and $Y,Z$ are separated and of finite type over $S$.
	Then we have
	$$
		\alpha^*\beta^* = (\beta\circ\alpha)^*\colon G(Z)\to G(X).
	$$
	In particular, if $S$ is normal then the canonical transfer defines a structure of a presheaf with transfers over $S$ on $G$.
\end{theorem}

\begin{proof}
	Let $K$ be an algebraic closure of the function field of $X$ and let $\rho\colon \Spec K\to X$ be the canonical morphism.
	Repeatedly applying Lemma \ref{normal_arbitrary}, we get
	\begin{align*}
		\rho^*\alpha^* \beta^* 	&{}={} (\alpha\circ \rho)^*\beta^*		\\
								&{}={} (\beta\circ \alpha\circ \rho)^*	\\
								&{}={} \rho^*(\beta\circ\alpha)^*.		
	\end{align*}
	Since $\rho^*\colon G(X)\to G(K)$ is injective, this completes the proof.
\end{proof}

\section{Characterization}

In this section we characterize the canonical transfer by a simple condition on radicial transfers.
We continue to fix a noetherian base scheme $S$ and a separated commutative group scheme $G$ over $S$.
First we note the following remarkable property of the canonical transfer.

\begin{lemma}\label{lem:radicial}
	Let $X,Y$ be noetherian integral $S$-schemes and $V\subset X\times_SY$ be an integral closed subscheme finite flat radicial of degree $d$ over $X$.
	Suppose that $Y$ is separated and of finite type over $S$.
	Let $p\colon V\to X$ and $q\colon V\to Y$ be the canonical projections.
	Then for any $g\in G(Y)$, the element $d\cdot q^*g$ can be uniquely written as $p^*h$ for some $h\in G(X)$.
	Moreover, this $h$ coincides with $[V]^*g$.
\end{lemma}

We write $t_V(g)$ for the element $h$ appearing in the statement.

\begin{proof}
	The uniqueness follows from Lemma \ref{lem:dominant_injective}.
	We prove that the element $[V]^*g$ satisfies the condition $p^*[V]^*g = d\cdot q^*g$.
	Consider the following diagram, where the left square is Cartesian.
	$$
	\xymatrix{
		V\times_XV\ar[r]^-{\pr_2}\ar[d]^-{\pr_1}		&V\ar[r]	^-{q}\ar[d]^-{p}		&Y\\
		V\ar[r]^-{p}\ar@/^10pt/[u]^-{\Delta}						&X
	}
	$$
	We have $p^*[V]^*g = p^*p_*q^*g = {\pr_1}_*\pr_2^*q^*g$ by Lemma \ref{lem:base_change}, so it suffices to prove that
	$$
		{\pr_1}_*\pr_2^*g'=d\cdot g'
	$$
	holds for any $g'\in G(V)$.
	We claim that the diagonal morphism $\Delta\colon V\to V\times_XV$ satisfies
	$$
		{\pr_1}_*=d\cdot \Delta^*\colon G(V\times_XV)\to G(V).
	$$
	If this claim is proved, then evaluating at $\pr_2^*g'$ we get the desired equality.
	
	Let $K$ be an algebraic closure of the function field of $V$ and let $W=\Spec K\times_XV$.
	Then $W$ is finite radicial of degree $d$ over $\Spec K$, and has a section $\Delta_K\colon \Spec K\to W$ induced by $\Delta$.
	By Lemma \ref{lem:base_change} and the injectivity of $G(V)\to G(K)$, it suffices to prove
	$$
		{\pr_1}_*=d\cdot \Delta_K^*\colon G(W)\to G(K).
	$$	
	This follows from Lemma \ref{lem:reduction_scheme}.
\end{proof}

Actually, the above condition characterizes the canonical transfer.

\begin{theorem}
	Suppose that $S$ is normal and we are given a structure of a presheaf with transfers over $S$ on $G$; write $\alpha^\dagger$ for the map induced by a finite correspondence $\alpha$.
	Assume the following:
	\begin{quote}
		For any connected $X,Y\in \Sm_S$ and for any integral closed subscheme $V\subset X\times_SY$ finite flat radicial over $X$, we have $[V]^\dagger g = t_V(g)$ (see Lemma \ref{lem:radicial}).
	\end{quote}
	Then $\alpha^\dagger=\alpha^*$ holds for any finite correspondence $\alpha$.
\end{theorem}

\begin{proof}
	Let $X,Y \in \Sm_S$ be connected and $\alpha\in c_S(X,Y)$.
	We prove that $\alpha^\dagger = \alpha^*\colon G(Y)\to G(X)$ holds.
	Write $\alpha = \sum_{i=1}^nm_i[V_i]$.
	Let $K$ be a separable closure of the function field of $X$.
	Then each $V_i\times_X\Spec K$ is a disjoint union of finite radicial schemes over $\Spec K$, so we can write
	$$
		V_i\times_X\Spec K = \textstyle\coprod_{j=1}^{n_i} \Spec A_{ij}
	$$
	where $A_{ij}$ is a finite local $K$-algebra such that the residue field $L_{ij}$ is purely inseparable over $K$.
	We conclude by limit argument that there is some connected $U\in \Sm_S$ and an \'etale morphism $\pi\colon U\to X$ such that
	$$
		V_i\times_X U = \textstyle\coprod_{j=1}^{n_i} W_{ij},
	$$
	where $W_{ij}$ is an irreducible closed subscheme of $U\times_SY$ whose reduction is finite flat radicial over $U$.
	By Lemma \ref{lem:composition_formula} (2) and Lemma \ref{lem:pullback_lemma} we have $\alpha\circ \pi = \sum_{i=1}^n\sum_{j=1}^{n_i} m_{ij}[(W_{ij})_\red]$ for some $m_{ij}$, so we have $(\alpha\circ \pi)^\dagger = (\alpha\circ \pi)^*\colon G(Y)\to G(U)$ by our assumption.
	Since $\pi^\dagger = \pi^*\colon G(X)\to G(U)$ is injective by Lemma \ref{lem:dominant_injective}, we get $\alpha^\dagger = \alpha^*$.
\end{proof}

\printbibliography

\end{document}